\documentclass{amsart}

\usepackage{color}

\theoremstyle{plain} \newtheorem{theorem}{Theorem}[section]
\newtheorem{corollary}[theorem]{Corollary}
\newtheorem{lemma}[theorem]{Lemma}
\newtheorem{proposition}[theorem]{Proposition}

\theoremstyle{definition} \newtheorem{definition}[theorem]{Definition}
\newtheorem{remark}[theorem]{Remark}

\newtheorem{example}[theorem]{Example}
\newtheorem{notation}[theorem]{Notation}

\DeclareMathOperator{\Ric}{Ric} \DeclareMathOperator{\tr}{tr}
\DeclareMathOperator{\vol}{vol}

\newcommand{\Ent}{\mathcal{E}} \newcommand{\E}{\mathbb{E}}
\newcommand{\R}{\mathbb{R}} \newcommand{\N}{\mathbb{N}}
\newcommand{\PP}{\mathbb{P}}
\newcommand\newdot{{\kern.8pt\cdot\kern.8pt}}

\begin{document}

\title[Entropy on manifolds with time-dependent metric] {An entropy
  formula for the heat equation on manifolds with time-dependent
  metric, application to ancient solutions}

\author{Hongxin Guo} \author{Robert Philipowski} \author{Anton
  Thalmaier} \address{School of Mathematics and Information Science,
  Wenzhou University \hfill\break Wenzhou, Zhejiang 325035, China}
\email{\texttt{guo@wzu.edu.cn}} \address{Unit\'e de Recherche en
  Math\'ematiques, FSTC, Universit\'e du Luxembourg\hfill\break 6, rue
  Richard Coudenhove-Kalergi, L--1359 Luxembourg, Grand Duchy of
  Luxembourg} \email{\texttt{robert.philipowski@uni.lu},
  \texttt{anton.thalmaier@uni.lu}}

\thanks{Research supported by NSF of China (grants no. 11001203 and
  11171143) and Fonds National de la Recherche Luxembourg.}

\begin{abstract}
  We introduce a new entropy functional for nonnegative solutions of
  the heat equation on a manifold with time-dependent Riemannian
  metric.  Under certain integral assumptions, we show that this
  entropy is non-decreasing, and moreover convex if the metric evolves
  under super Ricci flow (which includes Ricci flow and fixed metrics
  with nonnegative Ricci curvature). As applications, we classify
  nonnegative ancient solutions to the heat equation according to
  their entropies. In particular, we show that a nonnegative ancient
  solution whose entropy grows sublinearly on a manifold evolving
  under super Ricci flow must be constant.  The assumption is sharp in
  the sense that there do exist nonconstant positive eternal solutions
  whose entropies grow exactly linearly in time. Some other results
  are also obtained.
\end{abstract}

\renewcommand{\theequation}{\arabic{section}.\arabic{equation}}
\maketitle

% \tableofcontents

\section{Introduction}
Let $M$ be a smooth manifold equipped with a family $(g(t))_{t\geq 0}$
of Riemannian metrics depending smoothly on $t$, and let $u$ be a
nonnegative solution of the backward heat equation
\begin{equation}\label{backwardheat}
  \frac{\partial u}{\partial t} + \Delta_{g(t)} u = 0. 
\end{equation}
The classical Boltzmann-Shannon entropy functional is defined by
\begin{equation*}
  \operatorname{Ent}(t) = -\int_M u(t,y) \log u(t,y) \vol_{g(t)}(dy)
\end{equation*}
(provided that the integral exists).  If the metric does not depend on
$t$, then, under certain reasonable assumptions on $u$ (for instance
if $u$ grows slowly enough so that integration by parts can be
justified) the Boltzmann-Shannon entropy is non-increasing, and
moreover concave if $\Ric \geq 0$.  In the case of a compact manifold
Lim and Luo~\cite{L-L} recently studied asymptotic estimates on the
time derivative of $\operatorname{Ent}$.  However, the classical
Boltzmann-Shannon entropy has two important drawbacks:
\begin{enumerate}
\item It need not be monotone if the metric depends on $t$.

\item On noncompact manifolds it is finite only for a relatively
  narrow class of functions.  Even if $u$ is a positive constant $\neq
  1$ on $\R^n$ (equipped with the standard metric), its
  Boltzmann-Shannon entropy equals $\pm\infty$.
\end{enumerate}

In this paper we introduce a new entropy functional of
Boltzmann-Shannon type which has much better chances to be finite and
which is monotone even if the metric depends on $t$.  We fix a point
$x \in M$ and let $p(t,x,y)$ be the heat kernel of the adjoint heat
equation
\begin{equation*}
  \frac{\partial p}{\partial t} = \Delta_{g(t)} p - \frac{1}{2} \tr \frac{\partial g}{\partial t} p.
\end{equation*}
In other words, $p(t,x, \cdot)$ is the density of $X_t$ with respect
to $\vol_{g(t)}$,
\begin{equation*}
  p(t, x, y)\vol_{g(t)}(dy)=\PP\{X_t\in dy\},
\end{equation*}
where $(X_t)_{t \geq 0}$ is a $(g(t))_{t \geq 0}$-Brownian motion
started at $x$ and speeded up by the factor $\sqrt{2}$ \cite{ACT,
  Coulibaly, kuwadaphilipowski}.  We assume that $\int_M p(t, x,
y)\vol_{g(t)}(dy) = 1$ for all $t>0$, in other words that Brownian
motion on $M$ does not explode.  By a result of Kuwada and the second
author~\cite{kuwadaphilipowski} this condition is satisfied, in
particular, if $(M, g(t))$ is complete for all $t \geq 0$ and the
metric evolves under backward super Ricci flow, i.e.
\begin{equation}\label{backwardsuperricci}
  \frac{\partial g}{\partial t} \leq 2 \Ric.
\end{equation}
We define the entropy of $u$ with respect to the heat kernel measure
$$p(t, x, y)\vol_{g(t)}(dy)$$ by
\begin{align}
  \Ent(t)  :&=  \E \left[ (u \log u)(t, X_t) \right] \nonumber \\
  & = \int_M (u \log u)(t, y) p(t, x,
  y)\vol_{g(t)}(dy). \label{entropy}
\end{align}
(From a physical point of view it would be more natural to call
$-\Ent$ entropy. Our sign convention has the advantage of avoiding
unnecessary minus signs.)  Note that in contrast to the classical
Boltzmann-Shannon entropy this entropy is well-defined for all
non-negative solutions $u$ (because the heat kernel has total mass 1
and the function $u \mapsto u \log u$ is bounded from below).
Moreover, thanks to the fast decay of the heat kernel our entropy is
finite in most cases of interest.  In the next section we will show
that under certain integral assumptions $\Ent(t)$ is non-decreasing,
and moreover convex in the case of super Ricci flow.

\begin{remark}\label{timereversal}
  If we apply the substitution $\tau := -t$, then \eqref{backwardheat}
  and \eqref{backwardsuperricci} become
  \begin{equation*}
    \frac{\partial u}{\partial \tau} = \Delta_{g(-\tau)} u
  \end{equation*}
  and
  \begin{equation}\label{superricci}
    \frac{\partial g}{\partial \tau} \geq -2 \Ric.
  \end{equation}
  In other words, with respect to the new time variable $\tau$, the
  function $u$ is a solution to the forward heat equation, and $g$
  evolves according to super Ricci flow. In particular, solutions of
  \eqref{backwardheat} that are defined for all $t \geq 0$ are the
  same as ancient solutions of the heat equation.
\end{remark}

The most important examples of super Ricci flow are of course the
Ricci flow itself \cite{chowetal, chowknopf, chowluni, morgantian,
  toppinglectures}, where \eqref{superricci} holds with equality, and
fixed metrics with nonnegative Ricci curvature. Other interesting
examples are the extended Ricci flow introduced by List~\cite{list}
and Ricci flow coupled with harmonic map flow, as studied by
M\"uller~\cite{muller2}.

Ancient solutions to the heat equation are generalizations of harmonic
functions. Yau's Liouville theorem for positive harmonic functions
states that any positive harmonic function on a noncompact manifold
with nonnegative Ricci curvature is constant \cite{Y}. However, as we
can see from the example $u(\tau,y)=e^{\tau+y}$, Yau's Liouville
theorem cannot be generalized to positive ancient solutions without
any further assumptions. Based on this observation, Souplet and
Zhang~\cite[Theorem~1.2]{S-Z} proved the following: Let $M$ be a
complete, noncompact manifold with a fixed metric of nonnegative Ricci
curvature.  If $u$ is a positive ancient solution to the heat equation
such that $\log u(\tau, y) = o(d(y)+\sqrt{|\tau|})$ near infinity,
then $u$ must be constant.

Souplet and Zhang's assumption is a pointwise one. They proved their
result by establishing a sharp gradient estimate for the heat
equation, which has many other important applications.  There has been
some recent work concerning ancient solutions for heat or more general
diffusion equations, for instance \cite{W} and \cite{Z}. A part of
these results consists in imposing certain pointwise growth
assumptions on the ancient solutions and using various gradient
estimates to conclude that such solutions must be constant.

It is desirable to give an integral assumption besides the pointwise
assumption. As an application of our entropy formula, among other
results we prove the following: Assume that $\frac{\partial
  g}{\partial t} \leq 2 \Ric$, and let $u$ be a nonnegative solution
of \eqref{backwardheat}.  If its entropy $\Ent(t)$ grows sublinearly,
i.e.\ $\lim_{t\to\infty} \Ent(t)/t = 0$, then $u$ is constant.

Our assumption that $\Ent(t)$ grows sublinearly is an integral one.
The result is sharp in the sense that there do exist nonconstant
ancient solutions whose entropies grow linearly, for instance $u(t,y)
= e^{y-t}$ whose entropy with parameter $x = 0$ satisfies $\Ent(t)=t$.

We also discuss the special case when $\Ent(t)$ is a linear function
of $t$. In this case under the assumption that $\frac{\partial
  g}{\partial t} \leq 2 \Ric$, we show that $u$ is the product of a
function depending only on $t$ and a function depending only on $y$.

\section{Monotonicity and convexity of the entropy}
In this section we derive formulas for the first two variations of the
entropy. We shall see that the entropy $\Ent(t),$ under certain
assumptions, is non-decreasing. Moreover $\Ent(t)$ is convex if
$\frac{\partial g}{\partial t} \leq 2 \Ric$.

\begin{theorem}\label{derivatives of entropy}
  Let $u$ be a solution of the backward heat
  equation~\eqref{backwardheat}. Suppose that for $t>0$,
  \begin{equation}\label{cond1}
    \int_M |\nabla(u \log u)|^2(t,y)\,p(t,x,y) \vol_{g(t)}(dy)< \infty
  \end{equation}
  and
  \begin{equation}\label{cond2}
    \int_M \left| \nabla \left( \frac{|\nabla u|^2}{u} \right) \right|^2(t,y)\,
    p(t,x,y)\vol_{g(t)}(dy) < \infty.
  \end{equation}
  Then as long as $\Ent(t)$ is finite its first derivative is given by
  \begin{equation}\label{1st derivative of entropy}
    \Ent'(t) = \int_M \frac{|\nabla u|^2}{u} \left(t, y \right)\, p(t,x,y) \vol_{g(t)}(dy),
  \end{equation}
  and its second derivative by
  \begin{align}
    \Ent''(t) = \int_M \left(2u\left(|\nabla\nabla\log u|^2 + \left(
          \Ric - \frac{1}{2} \frac{\partial g}{\partial
            t} \right)
        (\nabla \log u, \nabla \log u) \right) \right) \left( t, y \right)&\notag \\
    p( t, x, y)\,\vol_{g(t)}(dy).&
    \label{2nd derivative of entropy}
  \end{align}
\end{theorem}

For the proof we need the following lemma:

\begin{lemma}\label{derivatives}
  If $u$ solves the backward heat equation~\eqref{backwardheat}, we
  have
  \begin{equation}\label{ulogu1}
    \left( \frac{\partial}{\partial t} + \Delta_{g(t)} \right) (u \log u) = \frac{|\nabla u|^2}{u}
  \end{equation}
  and
  \begin{equation}\label{ulogu2}
    \left( \frac{\partial}{\partial t} + \Delta_{g(t)} \right) \left( \frac{|\nabla u|^2}{u} \right) = u \left( 2 |\nabla \nabla \log u|^2+ \left( 2 \Ric - \frac{\partial g}{\partial t} \right) (\nabla\log u, \nabla \log u) \right).
  \end{equation}
\end{lemma}

\begin{proof}
  The first equality is straight-forward. The second one is well-known
  in the case of a fixed Riemannian metric (e.g.~\cite{Hamilton}; for
  a proof see \cite[Proposition~2.1]{L-L}). The additional term
  $-\frac{\partial g}{\partial t}(\nabla\log u, \nabla \log u)$
  appearing here comes from the time-derivative of $|\nabla u|^2$ via
  the formula
  \begin{equation}
    \label{MinusSign}
    \frac{\partial}{\partial t} \left( |\nabla f|^2 \right) =  -\frac{\partial g}{\partial t}(\nabla f, \nabla f),\quad
    f\in C^\infty(M).
  \end{equation}
  Note that not only $|\cdot|$, but also $\nabla$ depends on $t$,
  which is the reason for the minus sign in formula~\eqref{MinusSign}.
\end{proof}

Using Lemma~\ref{derivatives} it is easy to give a formal proof of
Theorem~\ref{derivatives of entropy} via integration by parts.
However, since $M$ is not assumed to be compact, the feasibility of
integration by parts is difficult to justify, and therefore we present
a proof based on stochastic analysis. In this proof the
assumptions~\eqref{cond1} and \eqref{cond2} are used to show that
certain local martingales are indeed true martingales.  One should
note that thanks to the exponential decay of the heat kernel (see
e.g.\ \cite{chengliyau} for the case of a fixed metric and
\cite[Section~6.5]{zhang} for the case of Ricci flow) \eqref{cond1}
and \eqref{cond2} are satisfied in most cases of interest.

\begin{remark}
  In terms of a $(g(t))_{t \geq0}$-Brownian motion $(X_t)_{t \geq 0}$
  started at $x$, conditions \eqref{cond1} and \eqref{cond2} read as
  \begin{align}
    &\E\left[|\nabla(u \log u)|^2(t,X_t)\right] < \infty,\label{cond1a}\\
    &\E\left[\left| \nabla \left( \frac{|\nabla u|^2}{u} \right)
      \right|^2(t,X_t)\right] < \infty,
    \label{cond2a}
  \end{align}
  and imply that
  \begin{align*}
    \int_0^t \E\left[|\nabla(u \log u)|^2(s,X_s)\right]ds < \infty\
    \text{ and }\ \int_0^t \E\left[\left| \nabla \left( \frac{|\nabla
            u|^2}{u} \right) \right|^2(s,X_s)\right]ds < \infty,
  \end{align*}
  which are standard conditions to assure that the martingale parts of
  the processes
  \begin{align*}
    (u \log u)(s,X_s),\quad \frac{|\nabla u|^2}{u} (s,X_s), \quad
    0\leq s\leq t,
  \end{align*}
  are true martingales (even $L^2$-martingales).  The condition,
  analogous to \eqref{cond1a} resp. \eqref{cond2a}, guaranteeing that
  the local martingale
  \begin{align*}
    u(s,X_s),\quad 0\leq s\leq t,
  \end{align*}
  is a true $L^2$-martingale reads as
  \begin{align}
    &\E\left[|\nabla u |^2(t,X_t)\right] < \infty.\label{cond0a}
  \end{align}
\end{remark}

\begin{proof}[Proof (of Theorem~\ref{derivatives of entropy})]
  Let $U$ be a horizontal lift of the $(g(t)_{t \geq 0}$-Brownian
  motion $X$ and $Z$ the corresponding anti-development of $X$ ($Z$ is
  an $\R^d$-valued Brownian motion speeded up by the factor
  $\sqrt{2}$).  By It\^o's formula (see
  \cite[Lemma~1]{kuwadaphilipowski})
  \begin{equation}\label{ito1}
    d(u \log u)(t, X_t) = \left( \frac{\partial}{\partial t} + \Delta_{g(t)} \right) (u \log u)(t, X_t) dt + dM_t,
  \end{equation}
  where thanks to \eqref{cond1} the local martingale
  \begin{equation*}
    M_t := \sum_{i=1}^d (d(u \log u))(s,X_s)(U_s e_i) dZ_s^i
  \end{equation*}
  is a true martingale (as stochastic integral of a square-integrable
  process, see e.g.\ \cite[Definition~3.2.9]{karatzas}). Combining
  \eqref{ito1} and \eqref{ulogu1} we obtain
  \begin{equation*}
    \E \left[ (u \log u)(t, X_t) \right] = (u \log u)(0, x) + \E \left[ \int_0^t \frac{|\nabla u|^2}{u}(s, X_s) ds \right]
  \end{equation*}
  or, in other words,
  \begin{equation*}
    \Ent(t) = \Ent(0) + \int_0^t \E \left[ \frac{|\nabla u|^2}{u}(s, X_s) \right] ds,
  \end{equation*}
  and hence
  \begin{equation*}
    \Ent'(t) =  \E \left[ \frac{|\nabla u|^2}{u}(t, X_t) \right] = \int_M \frac{|\nabla u(t,y)|^2}{u(t, y)} p(t,x,y) \vol_{g(t)}(dy),
  \end{equation*}
  as claimed.

  The formula for the second derivative of $\Ent$ can be proved in the
  same way, using \eqref{cond2} and~\eqref{ulogu2}.
\end{proof}

\section{Gradient-entropy estimates}\label{Section:Gradient-entropy
  estimates}

In this section we give gradient estimates for positive solutions of
the backward heat equation \eqref{backwardheat} in terms of their
entropy.

\begin{proposition}\label{entropy estimate}
  Let $u: {[0,T]}\times M \to \R_+$ be a positive solution of the
  backward heat equation \eqref{backwardheat} satisfying the
  conditions \eqref{cond1a}, \eqref{cond2a} and \eqref{cond0a} for
  $t=T$.  Assume that $\frac{\partial g}{\partial t} \leq 2 \Ric$.
  Then, for each $t\in{]0,T]}$ and $x\in M$,
  \begin{align}
    t\,\left|\frac{\nabla u}{u}\right|^2
    (0,x)\leq\E\left[\frac{u(t,X_t)}{u(0,x)}\log\frac{u(t,X_t)}{u(0,x)}\right]
  \end{align}
  where $(X_s)$ is a $(g(s)_{s\geq0}$-Brownian motion starting at
  $x$. In other words, if $u$ is normalized such that $u(0,x)=1$, then
  \begin{align}
    |\nabla u|^2 (0,x)\leq\frac{\Ent(t)}t.
  \end{align}
\end{proposition}

\begin{proof}
  Consider the process
  \begin{equation}
    \label{N_t}
    N_s:=(t-s)\,\frac{|\nabla u|^2}{u} (s,X_s)  
    +\big(u\log u\big)(s,X_s),\quad 0\leq s\leq t,
  \end{equation}
  which is easily seen to be a submartingale under the given
  conditions.  This enables us to exploit the inequality
  $\E[N_0]\leq\E[N_t]$ which gives
  \begin{equation}
    \label{grad_entropy_estimate}
    t\,\frac{|\nabla u|^2}{u} (0,x)+\big(u\log u\big)(0,x)\leq \E\left[\big(u\log u\big)(t,X_t)\right].
  \end{equation}
  Combining this with the fact that $u(0,x)=\E[u(t,X_t)]$ which
  follows from the martingale property of $(u(s,X_s))_{0\leq s\leq
    t}$, the claimed inequality is obtained.
\end{proof}

\begin{corollary}
  Let $u$ be a positive solution of the backward heat equation
  \eqref{backwardheat} on ${[0,T]}\times M$.  We keep the assumptions
  of Theorem \textup{\ref{entropy estimate}}. Let $x\in M$ and
  $0<t\leq T$.
  \begin{enumerate}
  \item Then, for any $\delta>0$,
    \begin{align*}
      \left|\frac{\nabla u}{u}\right|^2 (0,x)\leq\frac\delta{2t}+
      \frac1{2\delta}\,\E\left[\frac{u(t,X_t)}{u(0,x)}\log\frac{u(t,X_t)}{u(0,x)}\right]
    \end{align*}
  \item If $m:=\sup_{{[0,t]}\times M}u$, then
    \begin{equation*}
      \frac{\vert\nabla u\vert}{u}(0,x)\leq\frac1{t^{1/2}}\sqrt{\log\frac m{u(0,x)}}\,.
    \end{equation*}
  \end{enumerate}
\end{corollary}

\section{Entropy and linear growth}\label{Section:Entropy and linear growth}
We now investigate positive solutions of the backward heat equation
\eqref{backwardheat} according to their entropy.  Recall that by
Remark~\textup{\ref{timereversal}\/} any global solution to the
backward equation \eqref{backwardheat} gives rise to an ancient
solution of the forward heat equation.

\begin{theorem}\label{entropy grows sublinearly}
  Let $u: \R_+ \times M \to \R_+$ be a positive solution of the
  backward heat equation \eqref{backwardheat} satisfying \eqref{cond1}
  and \eqref{cond2} for all $t > 0$.  If $\frac{\partial g}{\partial
    t} \leq 2 \Ric$ and if the entropy of $u$ grows sublinearly, i.e.\
  $\lim_{t\to\infty} \Ent(t)/t=0$, then $u$ is constant.
\end{theorem}

\begin{proof}
  Since $\Ent$ is convex, the condition
  $\lim_{t\to\infty}\frac{\Ent(t)}t = 0$ implies that $\Ent$ is
  constant. Therefore
  \begin{equation*}
    \Ent'(t) = \E \left[ \frac{|\nabla u|^2}{u} \left(t, X_t \right) \right] \equiv 0,
  \end{equation*} 
  so that $u$ is constant.
\end{proof}

\begin{remark} 
  Note that Theorem \ref{entropy grows sublinearly} also immediately
  follows from the results of the last section. Indeed, if the
  conditions \eqref{cond1} and \eqref{cond2}, or equivalently, the
  conditions \eqref{cond1a}, \eqref{cond2a} hold for all $t>0$, we
  have that \eqref{N_t} is a true submartingale. Rewriting
  $\E[N_0]\leq\E[N_t]$, resp.~$\E[N_{t/2}]\leq\E[N_t]$, we have
  \begin{align*}
    &\frac{|\nabla u|^2}{u} (0,x)+\frac1t\,\big(u\log u\big)(0,x)\leq
    \frac1t\,\E\left[\big(u\log u\big)(t,X_t)\right],\quad\text{resp.,}\\
    &\E\left[\frac{|\nabla u|^2}{u} (t/2,X_{t/2})\right]
    +\frac2t\,\E\left[\big(u\log u\big)(t/2,X_{t/2})\right]\leq
    \frac2t\,\E\left[\big(u\log u\big)(t,X_t)\right],
  \end{align*}
  and it suffices to take the limit as $t\to\infty$.
\end{remark}

\begin{remark} Let $\frac{\partial g}{\partial t} \leq 2 \Ric$. For a
  positive solution $u: \R_+ \times M \to \R_+$ of the backward heat
  equation \eqref{backwardheat} we may consider the constant
$$\theta:=\lim_{t\to\infty}\Ent'(t)$$
which is well-defined by the monotonicity resulting from formula
\eqref{2nd derivative of entropy}.  The value of $\theta$ may be zero,
a positive constant or $+\infty$.  Theorem \ref{entropy grows
  sublinearly} can then be rephrased to the statement that a positive
solution $u: \R_+ \times M \to \R_+$ of the backward heat equation
\eqref{backwardheat}, satisfying \eqref{cond1} and \eqref{cond2} for
all $t \geq 0$, is trivial if and only if $\theta=0$.
\end{remark}

\begin{theorem}\label{linear entropy means separation}
  Let $u: [0,T] \times M \to \R_+$ be a positive solution of the
  backward heat equation \eqref{backwardheat} satisfying \eqref{cond1}
  and \eqref{cond2} for $t=T$. If $\frac{\partial g}{\partial t} \leq
  2 \Ric$ and $\Ent(t)$ is an exactly linear function of $t$, then $u$
  has the form
  \begin{equation}\label{u seperation}
    u(t, y) = \psi(y)\phi(t) 
  \end{equation}
  for some functions $\psi$ and $\phi$. Moreover, $\psi$ and $\phi$
  satisfy the differential equation
  \begin{equation}\label{phi-psi equation}
    \frac{\partial \phi / \partial t}{\phi} = -\frac{\Delta \psi}{\psi}.
  \end{equation}
\end{theorem}

\begin{proof}
  Since $\Ent(t)$ is exactly linear, we have
  \begin{align*}
    \Ent''(t) &= \E\left[2u\left(|\nabla\nabla\log u|^2+\left(
          \Ric
          - \frac{1}{2} \frac{\partial g}{\partial t} \right)(\nabla\log u, \nabla \log u)\right)(t,X_t)\right]\\
    &= \int_M 2u\left(|\nabla\nabla\log u|^2+\left( \Ric
        - \frac{1}{2} \frac{\partial g}{\partial t} \right)(\nabla\log
      u, \nabla \log u)\right)p dy \equiv 0.
  \end{align*}
  In particular, $\nabla\nabla\log u\equiv 0$ implies that
  \begin{equation*}
    \nabla \left( \left| \nabla \log u \right|^2 \right) = 2\nabla\nabla\log u(\nabla\log u,\cdot)=0
  \end{equation*}
  so that $|\nabla \log u|$ is a function of $t$ only. Since
  \begin{align*}
    0=\operatorname{tr}(\nabla\nabla\log u)=\Delta\log u =
    \frac{\Delta u}{u}-\frac{|\nabla u|^2}{u^2} = -\frac{\partial\log
      u}{\partial t} - |\nabla\log u|^2,
  \end{align*}
  we have
  \begin{equation*}
    \frac{\partial\log u}{\partial t} = -|\nabla\log u|^2
  \end{equation*}
  and hence
  \begin{equation*}
    \log u(t,y) - \log u(0, y) = -\int_0^t |\nabla \log u|^2(s)ds.
  \end{equation*}
  It follows that
  \begin{equation*}
    u(t, y) = u(0,y) \exp \left( -\int_0^t |\nabla\log u|^2(s)ds \right),
  \end{equation*}
  and this proves Eq.~\eqref{u seperation}.  Then Eq.~\eqref{phi-psi equation}
  follows immediately from \eqref{u seperation} and the backward heat
  equation for $u$.
\end{proof}

We also have the following simple observation:
\begin{proposition}\label{positivedefinite}
  If $2 \Ric - \frac{\partial g}{\partial t}$ is positive definite
  everywhere, then no nonconstant positive solution to the backward
  heat equation satisfying \eqref{cond1} and \eqref{cond2} can have
  linear entropy.
\end{proposition}
\begin{proof}
  If $\Ent(t)$ is linear, we have $\Ent''(t)\equiv0$ which implies
  that
  \begin{equation*}
    \left( 2 \Ric - \frac{\partial g}{\partial t} \right) (\nabla\log u,\nabla\log u) \equiv 0,
  \end{equation*}
  as the heat kernel $p(t, x, y)$ is strictly positive everywhere.
  Since $ 2 \Ric - \frac{\partial g}{\partial t}$ is strictly positive
  everywhere, we get $\nabla\log u \equiv 0$.
\end{proof}

\begin{example}
  Let $u(t, y)=e^{y-t}$ on $\R$ equipped with the standard
  metric. Choose $x=0$ in the heat kernel, so that
  $p(t,x,y)=e^{-y^2/4t}/\sqrt{4\pi t}$. We first check that the
  assumptions \eqref{cond1} and \eqref{cond2} are satisfied: By
  elementary and straightforward calculations, for every $t>0$ we have
  \[
  \int_{\mathbb R} |\nabla(u \log u)|^2(t,y) p(t,x,y) dy = (9t^2 + 8t
  + 1)e^{2t} < \infty
  \]
  and
  \[
  \int_{\mathbb R} \left| \nabla \left( \frac{|\nabla u|^2}{u} \right)
  \right|^2(t,y) p(t,x,y) dy = e^{2t} < \infty.
  \]
  An easy calculation shows that
  \begin{align*}
    \Ent (t)&=\frac 1{\sqrt{4\pi t}}\int_{\mathbb R}(-t+y)e^{-\frac {y^2}{4t}+y-t}dy\\
    &=\frac 1{\sqrt{4\pi t}}\int_{\mathbb
      R}\left((y-2t)+t\right)e^{-\frac {1}{4t}(y-2t)^2}dy=t,
  \end{align*}
  so that $\Ent$ grows exactly linearly.

  More generally, for any constants $a>0$ and $b$, the function $u(t,
  y)=a e^{by-b^2t}$ is a positive solution of the backward heat
  equation, and its entropy $\Ent(t)=a(\log a + b^2 t)$ grows exactly
  linearly. Similar examples can be constructed on $\R^n$.
\end{example}

\begin{example}
  Let $M=\R^3\setminus\{0\}$ be equipped with the standard metric. The
  function
  \begin{align*}
    u(x)=\frac1{\|x\|}
  \end{align*}
  is harmonic on $M$, and thus $u(t,x)\equiv u(x)$ provides trivially
  a stationary solution of the backward heat equation
  \eqref{backwardheat} with respect to the static Euclidean
  metric. Let $X$ be a Brownian motion on $\R^3$ (with generator
  $\Delta$) starting at $e_1=(1,0,0)$.  Since
  \begin{align*}
    \nabla u(x)=-\frac x{\|x\|^3}\quad\text{and}\quad\frac{|\nabla
      u|^2}u(x)=\frac1{\|x\|^3},
  \end{align*}
  it is easy to check that
  \begin{align*}
    \Ent(t) = \E \left[ (u \log u)(X_t) \right] = -\frac1{(4\pi
      t)^{3/2}}\int_{\R^3\setminus\{0\}}\frac{\log\|x\|}{\|x\|}\,\exp\left(-\frac{\|x-e_1\|^2}{4t}\right)
    dx,
  \end{align*}
  which is clearly bounded as function of $t$, whereas
  \begin{align*}
    \E \left[ \frac{|\nabla u|^2}{u}(X_t) \right] = \frac1{(4\pi
      t)^{3/2}}\int_{\R^3\setminus\{0\}}\frac1{\|x\|^3}\,\exp\left(-\frac{\|x-e_1\|^2}{4t}\right)
    dx=+\infty.
  \end{align*}
  This shows that Theorem \ref{derivatives of entropy} fails in
  general without assumptions on $u$.  Since the entropy $\Ent(t)$ of
  $u$ grows sublinearly, it also shows that Theorem \ref{entropy grows
    sublinearly} fails without extra assumptions, like the conditions
  \eqref{cond1} and \eqref{cond2}.
 
\end{example}

\section{Monotonicity and convexity of local entropy}
The results presented in the previous sections depend on the technical
conditions \eqref{cond1} and \eqref{cond2} which have been used to
assure that certain local martingales are true martingales.  As
indicated, due to the fast decay of the heat kernel on complete
non-compact manifolds, the required conditions are rather weak.  In
this section we describe some ideas how Stochastic Analysis can be
used to localize the entropy.
 
\begin{definition}
  For a relatively compact domain $D \subset M$ we define local
  entropies as follows:
  \begin{equation*}
    \Ent_D(t) := \E \left[ (u \log u)(t \wedge \tau_D, X_{t \wedge \tau_D}) \right] \quad 
    \mbox{and} \quad \Ent_D := \E \left[ (u \log u)(\tau_D, X_{\tau_D}) \right],
  \end{equation*}
  where $\tau_D$ is the first exit time of $X$ from $D$ (with the
  convention $\tau_D := 0$ if $X$ does not start in~$D$).
\end{definition}

Note that since $(X_t)_{t\geq0}$ is an elliptic diffusion and since
$D$ is relatively compact, the stopping time $\tau_D$ is finite almost
surely. By Fatou's lemma we trivially have
$$\Ent_D\leq\lim_{t\to\infty}\Ent_D(t).$$

\begin{remark}
  As before let $X$ be a $(g(t))_{t\geq0}$-Brownian motion $X$.
  It\^o's formula, along with Eq.~\eqref{ulogu1}, implies
  \begin{equation*}
    \E \left[ (u \log u)(t \wedge \tau_D, X_{t \wedge \tau_D}) \right] 
    = (u \log u)(0, x) + \E \left[ \int_0^{t \wedge \tau_D} \frac{|\nabla u|^2}{u}(s, X_s) ds \right],
  \end{equation*}
  in other words,
  \begin{equation}
    \label{Ent_D(t) monoton}
    \Ent_D(t) 
    = \Ent_D(0) + \int_0^t \E \left[ \frac{|\nabla u|^2}{u}(s, X_s) \cdot 1_{\{s \leq \tau_D\}} \right] ds,
  \end{equation}
  and in particular,
  \begin{equation*}
    \Ent_D'(t) =  \E \left[ \frac{|\nabla u|^2}{u}(t, X_t) \cdot 1_{\{t \leq \tau_D\}} \right] \geq 0. 
  \end{equation*}
\end{remark}

\begin{notation}
  Equation \eqref{Ent_D(t) monoton} shows that $\Ent_D(t)$ is monotone
  both as a function of $t$ and $D$.  We define
  \begin{equation}
    \label{entropy_M}
    \Ent_M(t) := \lim_{D\uparrow M}\Ent_D(t)\equiv \Ent_D(0) 
    + \int_0^t \E \left[ \frac{|\nabla u|^2}{u}(s, X_s) \right] ds.
  \end{equation}
  Note that, as long as $\Ent_M(t)$ is finite, we always have
  \begin{equation*}
    \Ent_M'(t) := \E \left[ \frac{|\nabla u|^2}{u}(t,X_t) \right]
    \equiv\lim_{D\uparrow M}\Ent_D'(t).
  \end{equation*}
\end{notation}

\begin{remark}\label{remark_trivial}
  With a similar argument we have
  \begin{align*}
    &\lim_{n \to \infty} \E \left[ \frac{|\nabla u|^2}{u}(t \wedge
      \tau_{D_n}, X_{t \wedge \tau_{D_n}}) \right]
    = \frac{|\nabla u|^2}{u}(0, x)\\
    &\quad+ \int_0^t \E \left[ \left( 2u |\nabla\nabla\log u|^2+ 2u
        \left( \Ric - \frac{1}{2} \frac{\partial
            g}{\partial t} \right) (\nabla \log u, \nabla \log u)
      \right) (s, X_s) \right] ds.
  \end{align*}
\end{remark}

\begin{theorem}
  \label{thm:exactly_linear}
  Suppose that $\frac{\partial g}{\partial t} \leq 2 \Ric$, and let
  $(D_n)_{n \in \N}$ be an increasing sequence of relatively compact
  domains in $M$ satisfying $\cup_{n \in \N} D_n = M$.  Let $u: \R_+
  \times M \to \R_+$ be a positive solution of the backward heat
  equation \eqref{backwardheat} such that
  \begin{equation}
    \label{Condition_Bounded}
    \sup_{[0,t]\times M} \frac{|\nabla u|^2}{u}\leq C_t
  \end{equation}
  for each $t$ with a constant $C_t$ depending on $t$.  Then if the
  entropy of $u$ is of sublinear growth, i.e.
  \begin{equation}
    \label{E_M_sublinear}
    \frac{\Ent_M(t)}t\to0,\quad\text{as $t\to\infty$,}
  \end{equation}
  then $u$ is constant.
\end{theorem}

\begin{proof}
  Under condition \eqref{Condition_Bounded} the local submartingale
  \begin{align}
    \label{grad_u_submartingale}
    \frac{|\nabla u|^2}{u} (t,X_t), \quad t\geq0,
  \end{align}
  is bounded on compact time intervals, and hence is a true
  submartingale.  In particular, the expectations
  \begin{align*}
    t\to\E\left[\frac{|\nabla u|^2}{u} (t,X_t)\right]
  \end{align*}
  are non-decreasing. On the other hand, the condition
  \begin{equation*}
    \frac{\Ent_M(t)}t\equiv\frac{\Ent_M(0)}t+\frac1t\int_0^t \E\left[\frac{|\nabla u|^2}{u}(s, X_s) \right] ds
    \to0,\quad\text{as $t\to\infty$,}
  \end{equation*}
  implies that
  \begin{align*}
    \E\left[\frac{|\nabla u|^2}{u} (t_n,X_{t_n})\right]\to0
  \end{align*}
  for a sequence $t_n\uparrow\infty$. Hence,
$$\E\left[\frac{|\nabla u|^2}{u} (t,X_{t})\right]\equiv0$$
and consequently, $\nabla u(t,\newdot) \equiv 0$ for all $t$, so that
$u$ is constant in space. Since $u$ solves the backward heat equation,
this implies $\partial u/\partial t=0$ so that $u$ is constant in
space and time.
\end{proof}

\begin{theorem}
  Let $(D_n)_{n \in \N}$ be an increasing sequence of relatively
  compact domains in $M$ satisfying $\cup_{n \in \N} D_n = M$. Let $u:
  [0,T] \times M \to \R_+$ be a positive solution of the backward heat
  equation \eqref{backwardheat} such that
  \begin{equation}
    \label{Condition_Bounded1}
    \sup_{[0,T]\times M} \frac{|\nabla u|^2}{u}\leq C_T.
  \end{equation}
  Suppose that $t\mapsto\Ent_M(t)$ is a linear function on $[0,T]$.
  \begin{enumerate}
  \item If $\frac{\partial g}{\partial t} \leq 2 \Ric$, then $u$ has
    the form
    \begin{equation*}
      u(t, y) = \psi(y)\phi(t) 
    \end{equation*}
    for some functions $\psi$ and $\phi$. Moreover, $\psi$ and $\phi$
    satisfy the differential equation
    \begin{equation*}
      \frac{\partial \phi / \partial t}{\phi} = -\frac{\Delta \psi}{\psi}.
    \end{equation*}
  \item If $2 \Ric - \frac{\partial g}{\partial t}$ is positive
    definite everywhere, then $u$ is constant.
  \end{enumerate}
\end{theorem}

\begin{proof}
  Since $\Ent_M$ is linear, we get $$t\mapsto\E\left[ \frac{|\nabla
      u|^2}{u}(t, X_t) \right]$$ is a constant function. By Remark
  \ref{remark_trivial} we may conclude that
  \begin{equation*}
    \E \left[ \left( 2u |\nabla\nabla\log u|^2+ 2u \left( \Ric 
    - \frac{1}{2} \frac{\partial g}{\partial t} \right)(\nabla\log u,\nabla\log u)\right)(t, X_t)\right]=0.
  \end{equation*}
  One can now apply the same arguments as in the proofs of
  Theorem~\ref{linear entropy means separation} and
  Proposition~\ref{positivedefinite}.
\end{proof}

\begin{remark}
  For the results above, condition \eqref{Condition_Bounded} has been
  only used to assure that \eqref{grad_u_submartingale} is a true
  submartingale. In terms of $\tau_n:=\tau_{D_n}$ a necessary and
  sufficient condition for the true submartingale property is that
  \begin{align}
    \label{Kabanov-Stricker}
    \liminf_{n\to \infty}\E\left[ \frac{|\nabla u|^2}{u}(\tau_{n},
      X_{\tau_{n}})\,1_{\{\tau_{n}\leq t\}} \right]=0,
  \end{align}
  see~for instance~\cite{Kabanov-Stricker:2003}.
\end{remark}

\begin{remark}
  One may always write
  \begin{align}
    \label{trivial_splitting}
    \E\left[ \frac{|\nabla u|^2}{u}(t \wedge \tau_{n}, X_{t \wedge
        \tau_{n}}) \right] = \E\left[ \frac{|\nabla u|^2}{u}(t,
      X_{t})\,1_{\{t<\tau_{n}\}} \right] +\E\left[ \frac{|\nabla
        u|^2}{u}(\tau_{n}, X_{\tau_{n}})\,1_{\{\tau_{n}\leq t\}}
    \right]
  \end{align}
  where the left-hand-side of Eq.~\eqref{trivial_splitting} is
  monotone in $t$ and $n$.  Thus, if $u$ is of sublinear growth, by
  letting $t\to\infty$ in~\eqref{trivial_splitting}, monotonicity of
  $n\mapsto\E\left[ \frac{|\nabla u|^2}{u}(\tau_{n},
    X_{\tau_{n}})\right]$ is obtained (without extra conditions).  In
  the proof to Theorem \ref{thm:exactly_linear} we used however
  monotonicity along deterministic times, i.e.~monotonicity of
  $t\mapsto\E\left[ \frac{|\nabla u|^2}{u}(t, X_t)\right]$ which
  follows from Eq.~\eqref{trivial_splitting}, as $n\to\infty$, but
  under only the additional hypothesis~\eqref{Kabanov-Stricker}.
\end{remark}

% \subsection*{Acknowledgements}

\end{document}